\documentclass[reqno]{amsart}%
\usepackage{amssymb,setspace,mathtools}
\usepackage{ifpdf}
\ifpdf
 \usepackage[hyperindex]{hyperref}
\else
 \expandafter\ifx\csname dvipdfm\endcsname\relax
 \usepackage[hypertex,hyperindex]{hyperref}%
 \else
 \usepackage[dvipdfm,hyperindex]{hyperref}%
 \fi
\fi
\allowdisplaybreaks[4]
\numberwithin{equation}{section}
\theoremstyle{plain}
\newtheorem{thm}{Theorem}[section]
\theoremstyle{remark}
\newtheorem{rem}{Remark}[section]
\newcommand{\td}{\textup{d}}
\newcommand{\cmdeg}[1]{\sideset{}{_\textup{cm}^{#1}}\deg}

\begin{document}

\title[A double inequality for completely monotonic degree]
{A double inequality for completely monotonic degree of a remainder for an asymptotic expansion of the trigamma function}

\author[F. Qi]{Feng Qi}
\address{Institute of Mathematics, Henan Polytechnic University, Jiaozuo 454010, Henan, China\\
College of Mathematics and Physics, Inner Mongolia University for Nationalities, Tongliao 028043, Inner Mongolia, China\\
School of Mathematical Sciences, Tianjin Polytechnic University, Tianjin 300387, China}
\email{\href{mailto: F. Qi <qifeng618@gmail.com>}{qifeng618@gmail.com}, \href{mailto: F. Qi <qifeng618@hotmail.com>}{qifeng618@hotmail.com}, \href{mailto: F. Qi <qifeng618@qq.com>}{qifeng618@qq.com}}
\urladdr{\url{https://qifeng618.wordpress.com}}

\dedicatory{Dedicated to people facing and battling COVID-19}

\begin{abstract}
In the paper, the author presents a double inequality for completely monotonic degree of a remainder for an asymptotic expansion of the trigamma function. This result partially confirms one in a series of conjectures on completely monotonic degrees of remainders of asymptotic expansions for the logarithm of the gamma function and for polygamma functions.
\end{abstract}



\thanks{This paper was typeset using\AmS-\LaTeX}

\maketitle
\tableofcontents

\section{Simple preliminaries}
In the literature~\cite[Section~6.4]{abram}, the functions
\begin{equation*}
\Gamma(w)=\int_0^{\infty}s^{w-1}e^{-s}\td s\quad\text{and}\quad
\psi(w)=[\ln\Gamma(w)]'=\frac{\Gamma'(w)}{\Gamma(w)}
\end{equation*}
for $\Re(w)>0$ are known as the classical Euler gamma function and the digamma function. Moreover, the functions $\psi'(w)$, $\psi''(w)$, $\psi^{(3)}(w)$, and $\psi^{(4)}(w)$ are known as the tri-, tetra-, penta-, and hexa-gamma functions respectively. As a whole, all the derivatives $\psi^{(k)}(w)$ for $k\ge0$ are known as the polygamma functions~\cite{berg-pedersen-Alzer, Zhao-Chu-JIA-392431}.
\par
Recall from the chapters~\cite[Chapter~XIII]{mpf-1993}, \cite[Chapter~1]{Schilling-Song-Vondracek-2nd}, and~\cite[Chapter~IV]{widder} that, if $\phi(t)$ defines on an interval $I$ and satisfies
\begin{equation}\label{cmf-dfn-ineq}
(-1)^{k}\phi^{(k)}(t)\geq0
\end{equation}
for all $t\in I$ and $k\ge0$, then it is called a completely monotonic function on $I$.
Theorem~12b on~\cite[p.~161]{widder} reads that $\phi(t)$ is completely monotonic on $(0,\infty)$ if and only if
\begin{equation}\label{Theorem12b-Laplace}
\phi(t)=\int_0^\infty e^{-ts}\td\sigma(s)
\end{equation}
converges for all $t\in(0,\infty)$, where $\sigma(s)$ is nondecreasing on $(0,\infty)$. The integral representation~\eqref{Theorem12b-Laplace} is equivalent to say that $\phi(t)$ is completely monotonic on $(0,\infty)$ if and only if it is a Laplace transform of $\sigma(s)$ on $(0,\infty)$.
A result in~\cite[p.~98]{Dubourdieu} and~\cite[p.~395]{haerc1}\label{CM-not=0} asserts that, unless $\phi(t)$ is a trivial completely monotonic function, that is, a nonnegative constant on $(0,\infty)$, those inequalities in~\eqref{cmf-dfn-ineq} are all strict on $(0,\infty)$. Why do we investigate completely monotonic functions? One can find historic answers in two monographs~\cite{Schilling-Song-Vondracek-2nd,widder} and closely related references therein.
\par
Let $\phi(t)$ be defined on $(0,\infty)$ and $\phi(\infty)=\lim_{x\to\infty}\phi(t)$. If $t^r[\phi(t)-\phi(\infty)]$ is completely monotonic on $(0,\infty)$, but $t^{r+\varepsilon}[\phi(t)-\phi(\infty)]$ is not for any number $\varepsilon>0$, then $r$ is called the completely monotonic degree of $\phi(t)$ with respect to $t\in(0,\infty)$; if $t^r[\phi(t)-\phi(\infty)]$ is completely monotonic on $(0,\infty)$ for all $r\in\mathbb{R}$, then the completely monotonic degree of $\phi(t)$ with respect to $t\in(0,\infty)$ is said to be $\infty$. The notation $\cmdeg{t}[\phi(t)]$ was designed in~\cite{psi-proper-fraction-degree-two.tex} to denote the completely monotonic degree $r$ of $\phi(t)$ with respect to $t\in(0,\infty)$. Why do we investigate completely monotonic degrees? One can find significant answers in the second paragraph of~\cite[Remark~1]{CAM-D-18-02975.tex} or in the papers~\cite{AIMS-Math-2019595.tex, Qi-Agar-Surv-JIA.tex, 22ICFIDCAA-Filomat.tex} and closely related references therein.

\section{Motivations}
It is well-known in~\cite[p.~257, 6.1.40]{abram} and~\cite[p.~260, 6.4.11]{abram} that
\begin{equation}\label{ln-gamma-symp-eq}
\ln\Gamma(w)\sim\biggl(w-\frac12\biggr)\ln w-w+\frac12\ln(2\pi)+\sum_{k=1}^\infty \frac{B_{2k}}{2k(2k-1)w^{2k-1}}
\end{equation}
and
\begin{equation}\label{asymptotic-polypsi}
\psi^{(n)}(w)\sim(-1)^{n-1}\biggl[\frac{(n-1)!}{w^n}+\frac{n!}{2w^{n+1}}+\sum_{k=1}^\infty B_{2k}\frac{(2k+n-1)!}{(2k)!w^{2k+n}}\biggr]
\end{equation}
as $w\to\infty$ in $\lvert \arg w\rvert<\pi$ for $n\ge0$, where an empty sum is understood to be $0$ and $B_{2k}$ for $k\ge1$ are known as the Bernoulli numbers which can be generated~\cite{CAM-D-18-00067.tex} by
\begin{equation*}
\frac{w}{e^w-1}=1-\frac{w}2+\sum_{k=1}^\infty B_{2k}\frac{w^{2k}}{(2k)!}, \quad| w| <2\pi.
\end{equation*}
Stimulated by the asymptotic expansions~\eqref{ln-gamma-symp-eq} and~\eqref{asymptotic-polypsi}, many mathematicians considered and investigated the remainders
\begin{equation*}
R_n(t)=(-1)^n\Biggl[\ln\Gamma(t)-\biggl(t-\frac12\biggr)\ln t+t-\frac12\ln(2\pi) -\sum_{k=1}^n\frac{B_{2k}}{2k(2k-1)}\frac1{t^{2k-1}}\Biggr]
\end{equation*}
and its derivatives, where $n\ge0$. For detailed information, please refer to the papers~\cite{Allasia-Gior-Pecaric-MIA-02, merkle-jmaa-96}, \cite[Theorem~8]{psi-alzer}, \cite[Section~1.4]{Sharp-Ineq-Polygamma-Slovaca.tex}, \cite[Theorem~2]{Koumandos-jmaa-06}, \cite[Theorem~2.1]{Koumandos-Pedersen-09-JMAA}, \cite[Theorem~3.1]{Mortici-AA-10-134}, and closely related references therein.
\par
In~\cite[Section~4]{CAM-D-18-02975.tex} and~\cite{Mansour-Qi-CMD.tex-arXiv, Mansour-Qi-CMD.tex-HAL}, the author posed, modified, and finalized the following conjectures:
\begin{enumerate}
\item
the completely monotonic degrees
\begin{align}
\cmdeg{t}\biggl[\ln t-\frac{1}{2t}-\psi(t)\biggr]&=1, \label{pis-ln-deg=1}\\
\cmdeg{t}\biggl[\frac{1}{t}+\frac{1}{2t^2}+\frac{1}{6t^3}-\psi'(t)\biggr]&=3, \label{pis-ln-deg=3}\\
\cmdeg{t}\biggl[\psi'(t)-\biggl(\frac{1}{t}+\frac{1}{2t^2}+\frac{1}{6t^3}-\frac{1}{30t^5}\biggr)\biggr]&=4 \label{pis-ln-deg=5}
\end{align}
with respect to $t\in(0,\infty)$ are valid;
\item
when $m=0$, the completely monotonic degrees of $R_n(t)$ with respect to $t\in(0,\infty)$ satisfy
\begin{equation}\label{R-01-0th-deg}
\cmdeg{t}[R_0(t)]=0,\quad \cmdeg{t}[R_1(t)]=1,
\end{equation}
and
\begin{equation}\label{R-n-0th-deg}
\cmdeg{t}[R_n(t)]=2(n-1), \quad n\ge2;
\end{equation}
\item
when $m=1$, the completely monotonic degrees of $-R_n'(t)$ with respect to $t\in(0,\infty)$ satisfy
\begin{equation}\label{R0p=1-R1p=2}
\cmdeg{t}[-R_0'(t)]=1,\quad \cmdeg{t}[-R_1'(t)]=2,
\end{equation}
and
\begin{equation}\label{Qi-Conj-Gam-Deg-der1(2n-1)}
\cmdeg{t}[-R_n'(t)]=2n-1, \quad n\ge2;
\end{equation}
\item
when $m\ge2$, the completely monotonic degrees of $(-1)^mR_n^{(m)}(t)$ with respect to $t\in(0,\infty)$ satisfy
\begin{equation}\label{R0mR1m-conj-Eq}
\cmdeg{t}\bigl[(-1)^mR_0^{(m)}(t)\bigr]=m, \quad \cmdeg{t}\bigl[(-1)^mR_1^{(m)}(t)\bigr]=m+1,
\end{equation}
and
\begin{equation}\label{R-n-m-th-deg}
\cmdeg{t}\bigl[(-1)^mR_n^{(m)}(t)\bigr]=m+2(n-1), \quad n\ge2.
\end{equation}
\end{enumerate}
\par
In~\cite[Theorem~1]{psi-alzer}, \cite[Theorem~1]{theta-new-proof.tex-BKMS}, \cite[Theorem~2.1]{Mansour-Qi-CMD.tex-arXiv}, \cite[Theorem~2.1]{Mansour-Qi-CMD.tex-HAL}, and~\cite[Theorem~3]{Xu=Cen-Qi-Conj-JIA2020}, the first conjecture in~\eqref{R0p=1-R1p=2}, which is equivalent to~\eqref{pis-ln-deg=1}, was unconsciously or consciously verified again and again.
\par
In~\cite[Theorem~2.1]{Koumandos-Pedersen-09-JMAA}, it was proved that
\begin{equation*}
\cmdeg{t}\bigl[R_n(t)\bigr]\ge n, \quad n\ge0.
\end{equation*}
This result is weaker than those conjectures in~\eqref{R-n-0th-deg}.
\par
In~\cite[Theorem~1]{Chen-Qi-Srivastava-09.tex}, \cite[Theorem~2]{CAM-D-18-02975.tex}, and~\cite[Theorem~4]{Xu=Cen-Qi-Conj-JIA2020}, the second conjecture in~\eqref{R0p=1-R1p=2} was proved once again.
\par
In~\cite[Theorem~2.1]{Mansour-Qi-CMD.tex-arXiv} and~\cite[Theorem~2.1]{Mansour-Qi-CMD.tex-HAL}, those five conjectures in~\eqref{R-01-0th-deg}, \eqref{R0p=1-R1p=2}, and~\eqref{Qi-Conj-Gam-Deg-der1(2n-1)} were confirmed.
\par
In~\cite[Theorems~1 and~2]{Xu=Cen-Qi-Conj-JIA2020}, it was acquired that
\begin{equation*}
\cmdeg{t}\bigl[(-1)^2R_0''(t)\bigr]=2 \quad\text{and}\quad
\cmdeg{t}\bigl[(-1)^2R_1''(t)\bigr]=3.
\end{equation*}
These two results confirm conjectures in~\eqref{R0mR1m-conj-Eq} just for the case $m=2$. The latter is equivalent to the conjecture~\eqref{pis-ln-deg=3}.
\par
The main aim of this paper is to partially confirm the conjecture~\eqref{pis-ln-deg=5}, a special case $m=n=2$ of the conjecture~\eqref{R-n-m-th-deg}, by the double inequality
\begin{equation}\label{Rn2m2deg=4}
4\le\cmdeg{t}\bigl[(-1)^2R_2''(t)\bigr]\le5.
\end{equation}

\section{A double inequality and its proof}
The conjecture~\eqref{pis-ln-deg=5}, or say, a special case $m=n=2$ of the conjecture in~\eqref{R-n-m-th-deg}, can be partially confirmed by the following theorem.

\begin{thm}\label{pis-ln-deg=5-thm}
The completely monotonic degree of the function
\begin{equation*}
Q(t)=\psi'(t)-\frac{1}{t}-\frac{1}{2t^2}-\frac{1}{6t^3}+\frac{1}{30t^5}
\end{equation*}
with respect to $t\in(0,\infty)$ is not less that $4$ and not greater than $5$. In other words, the double inequality~\eqref{Rn2m2deg=4} is valid.
\end{thm}

\begin{proof}
Making use of the integral representations
\begin{equation*}
\frac1{w^\tau}=\frac1{\Gamma(\tau)}\int_0^\infty s^{\tau-1}e^{-ws}\td s, \quad \Re(w), \Re(\tau)>0
\end{equation*}
and
\begin{equation*}
\psi^{(k)}(w)=(-1)^{k+1}\int_{0}^{\infty}\frac{s^{k}}{1-e^{-s}}e^{-ws}\td s, \quad \Re(w)>0, k\ge1
\end{equation*}
in~\cite[p.~260, 6.4.1]{abram} and~\cite[p.~255, 6.1.1]{abram} and directly computing, we obtain
\begin{align*}
Q(t)&=\int_{0}^{\infty}\biggl(\frac{s}{1-e^{-s}}-1-\frac{s}{2}-\frac{s^2}{12}+\frac{s^4}{6!}\biggr)e^{-ts}\td s\\
&\triangleq\int_{0}^{\infty}h(s)e^{-ts}\td s,\\
h'(s)&=\frac{s^3+\bigl(s^3-30s+90\bigr)e^{2s}-2s\bigl(s^2+60\bigr)e^s-30s-90}{180(e^s-1)^2},\\
h''(s)&=\frac{e^{3s}\bigl(s^2-10\bigr)+3\bigl(s^2+20s+30\bigr)e^s-3\bigl(s^2-20s+30\bigr)e^{2s}-s^2+10}{60(e^s-1)^3},\\
h^{(3)}(s)&=\frac{se^{4s}+(90-34s)e^{3s}-114se^{2s}-2(17s+45)e^s+s}{30(e^s-1)^4},\\
h^{(4)}(s)&=\frac{\begin{bmatrix}e^{5s}+5(6s-25)e^{4s}+10(33s-35)e^{3s}\\
+10(33s+35)e^{2s}+5(6s+25)e^s-1\end{bmatrix}}{30(e^s-1)^5}\\
&=\frac{1}{30(e^s-1)^5}\sum_{k=7}^{\infty}\frac{\begin{bmatrix}
5^k+(110k-350)3^k+5(3k-50)2^{2k-1}\\
+(165k+350)2^{k} +30k+125\end{bmatrix}}{k!}s^k\\
&=\frac{1}{60(e^s-1)^5}\biggl(\frac{5s^7}{7}+\frac{25s^8}{14}+\frac{193s^9}{84}+\frac{85s^{10}}{42} +\frac{5065s^{11}}{3696}+\dotsm\biggr)\\
&>0
\end{align*}
on $(0,\infty)$. Consecutively utilizing the L'H\^ospital rule gives
\begin{equation*}
\lim_{s\to0^+}h(s)=0,\quad\lim_{s\to0^+}h'(s)=0,\quad\lim_{s\to0^+}h''(s)=0,\quad\lim_{s\to0^+}h^{(3)}(s)=0.
\end{equation*}
Accordingly, consecutively integrating by parts results in
\begin{align*}
t^4Q(t)&=-t^3\int_{0}^{\infty}h(s)\frac{\td\bigl(e^{-ts}\bigr)}{\td s}\td s\\
&=-t^3\biggl[h(s)e^{-ts}\big|_{s=0^+}^{s=\infty}-\int_{0}^{\infty}h'(s)e^{-ts}\td s\biggr]\\
&=t^3\int_{0}^{\infty}h'(s)e^{-ts}\td s\\
&=-t^2\int_{0}^{\infty}h'(s)\frac{\td\bigl(e^{-ts}\bigr)}{\td s}\td s\\
&=-t^2\biggl[h'(s)e^{-ts}\big|_{s=0^+}^{s=\infty}-\int_{0}^{\infty}h''(s)e^{-ts}\td s\biggr]\\
&=t^2\int_{0}^{\infty}h''(s)e^{-ts}\td s\\
&=-t\int_{0}^{\infty}h''(s)\frac{\td\bigl(e^{-ts}\bigr)}{\td s}\td s\\
&=-t\biggl[h''(s)e^{-ts}\big|_{s=0^+}^{s=\infty}-\int_{0}^{\infty}h^{(3)}(s)e^{-ts}\td s\biggr]\\
&=t\int_{0}^{\infty}h^{(3)}(s)e^{-ts}\td s\\
&=-\int_{0}^{\infty}h^{(3)}(s)\frac{\td\bigl(e^{-ts}\bigr)}{\td s}\td s\\
&=-\biggl[h^{(3)}(s)e^{-ts}\big|_{s=0^+}^{s=\infty}-\int_{0}^{\infty}h^{(4)}(s)e^{-ts}\td s\biggr]\\
&=\int_{0}^{\infty}h^{(4)}(s)e^{-ts}\td s.
\end{align*}
Consequently, the function $t^4Q(t)$ is completely monotonic on $(0,\infty)$. This means that
\begin{equation}\label{Q(t)ge4}
\cmdeg{t}[Q(t)]\ge4.
\end{equation}
\par
For $\Re(w)>0$ and $k\ge1$, we have the recurrent relation
\begin{equation*}
\psi^{(k-1)}(w+1)=\psi^{(k-1)}(w)+(-1)^{k-1}\frac{(k-1)!}{w^k}.
\end{equation*}
See~\cite[p.~260, 6.4.6]{abram}. From this, it follows that
\begin{equation}
\begin{aligned}\label{t2zero-lim-polyg}
\lim_{t\to0^+}\bigl[t^{k}\psi^{(k-1)}(t)\bigr]
&=\lim_{t\to0^+}\biggl(t^{k}\biggl[\psi^{(k-1)}(t+1)-(-1)^{k-1}\frac{(k-1)!}{t^{k}}\biggr]\biggr)\\
&=(-1)^{k}(k-1)!
\end{aligned}
\end{equation}
for $k\ge1$.
If $t^{4+\varepsilon}Q(t)$ for $\varepsilon\in\mathbb{R}$ were completely monotonic on $(0,\infty)$, by virtue of the result mentioned on page~\pageref{CM-not=0} and proved in~\cite[p.~98]{Dubourdieu} and~\cite[p.~395]{haerc1}, we see that the first derivative of $t^{4+\varepsilon}Q(t)$ should be negative on $(0,\infty)$, that is,
\begin{align*}
\bigl[t^{4+\varepsilon}Q(t)\bigr]'&=(4+\varepsilon)t^{3+\varepsilon}Q(t)+t^{4+\varepsilon}Q'(t)\\
&=Q(t)t^{3+\varepsilon}\biggl[(4+\varepsilon)+t\frac{Q'(t)}{Q(t)}\biggr]
\end{align*}
should be negative on $(0,\infty)$. Therefore, we acquire
\begin{equation*}
\varepsilon<-4-t\frac{Q'(t)}{Q(t)}
=-4-t\frac{\psi''(t)-\frac{1}{6t^6}+\frac{1}{2t^4}+\frac{1}{t^3}+\frac{1}{t^2}} {\psi'(t)+\frac{1}{30t^5}-\frac{1}{6t^3}-\frac{1}{2t^2}-\frac{1}{t}}
\to1, \quad t\to0^+
\end{equation*}
on $(0,\infty)$, where we used the limit~\eqref{t2zero-lim-polyg} in the last step. Consequently, we arrive at
\begin{equation}\label{Q(t)<5}
\cmdeg{t}[Q(t)]\le5.
\end{equation}
\par
Combining~\eqref{Q(t)ge4} with~\eqref{Q(t)<5} leads to the inequality~\eqref{Rn2m2deg=4}. The proof of Theorem~\ref{pis-ln-deg=5-thm} is complete.
\end{proof}

\section{An open problem}

For a given point $(a,b)$ on the first quadrant $\{(x,y):x,y>0\}$, let $y=f(x)$ be a differentiable function on $[0,a]$ with $a>0$, $f(0)=0$, and $f(a)=b>0$. Denote $\tan\theta(x)=f'(x)$ and
\begin{equation*}
S_{a,b}(f)=\int_{0}^{a}\cos\theta(x)\td x=\int_{0}^{a}\cos\arctan f'(x)\td x=\int_{0}^{a}\frac{1}{\sqrt{1+[f'(x)]^2}\,}\td x.
\end{equation*}
What is the minimum
$$
\inf\bigl\{S_{a,b}(f):f(0)=0,f(a)=b,f(x)\in C^1([0,a])\bigl\}?
$$
What is the differentiable function $f_m(x)$ on $[0,a]$ such that $S_{a,b}(f_m)$ attains the above infimum?

\begin{rem}
This paper is a slightly revised version of the preprint~\cite{CM-degree-tri-p5.tex}.
\end{rem}

\end{document}